\documentclass[reqno]{amsproc}
\usepackage{aliascnt}
\usepackage{hyperref}
\usepackage{amsthm,amsmath,amsfonts,amscd,mathrsfs, diagrams}
\usepackage{bbm}

\hypersetup{pdftitle={A Landau Ginzburg Mirror Symmetry Isomorphism}, pdfauthor={Amanda Francis, Tyler Jarvis, Drew Johnson, Rachel Suggs}}

\DeclareMathOperator{\hess}{hess}
\newcommand{\BB}{\mathscr B} 
\newcommand{\CC}{\mathbb{C}}
\newcommand{\Gmax}{G^{\mathrm{max}}}
\newcommand{\HH}{\mathscr A} 
\newcommand{\II}{\mathscr I} 

\newcommand{\MM}{\overline{\mathscr{M}}}
\newcommand{\M}{A} 
\newcommand{\QQQ}{\mathscr Q}
\newcommand{\QQ}{\mathbb Q}
\newcommand{\SL}{\operatorname{S{\kern-.05em}L}}
\newcommand{\W}{\overline{\mathscr{W}}}
\newcommand{\X}{\mathbf X}
\newcommand{\Y}{\mathbf Y}
\newcommand{\ZZ}{\mathbb Z}
\newcommand{\adm}{a}%

\newcommand{\bTheta}{\boldsymbol \Theta}

\newcommand{\ba}{\mathbf a}
\newcommand{\expon}{\mathbf{h}_W}
\newcommand{\expontr}{\mathbf{h}_{W\tr}}
\newcommand{\bbeta}{\boldsymbol \beta}

\newcommand{\bb}{\mathbf b}
\newcommand{\bgamma}{\boldsymbol \gamma}
\newcommand{\bg}{\mathbf g}
\newcommand{\bhbar}[1]{{\bh}|_{\fix(g_{#1})}}
\newcommand{\bh}{\mathbf h}

\newcommand{\bn}{\mathbf n}

\newcommand{\brr}{\mathbf r}
\newcommand{\bstar}{\star_{B}}
\newcommand{\bs}{\mathbf s}
\newcommand{\bt}{\mathbf t}

\newcommand{\bv}{\mathbf v}
\newcommand{\bwa}{\mathbf w} 
\newcommand{\bwb}{\bwa'} 

\newcommand{\dsand}{{\quad \text{ and }\quad}}
\newcommand{\dseo}[2]{\sum_{#2} #1}
\newcommand{\dse}[1]{\sum #1}
\newcommand{\fix}{\operatorname{Fix}}

\newcommand{\kn}[2]{ \left\lceil #1 \:; #2 \right\rfloor }
\newcommand{\kna}[2]{{\kn{#1}{#2}}^A}
\newcommand{\knb}[2]{{\kn{#1}{#2}}^B}
\newcommand{\mmap}{\varphi} 

\newcommand{\one}{\mathbf 1}
\newcommand{\pair}[2]{\left< #1,\; #2 \right>}
\newcommand{\rca}{k} 
\newcommand{\rcb}{\ell} 
\newcommand{\ringid}{\mathbbm{1}}
\newcommand{\rr}{r} 
\newcommand{\st}{\operatorname{st}}
\newcommand{\stara}{\star_A}
\newcommand{\starb}{\star_B}

\newcommand{\tr}{^{\mathsf T}}
\newcommand{\uBB}{\uHH} 
\newcommand{\uHH}{\mathscr{H}} 
\newcommand{\wa}{w} 
\newcommand{\wb}{w'} 

\providecommand*{\lemautorefname}{Lemma}
\providecommand*{\remarkautorefname}{Remark}
\providecommand*{\corautorefname}{Corollary}
\providecommand*{\axautorefname}{Axiom}
\providecommand*{\propertyautorefname}{Property}
\providecommand*{\theoremautorefname}{Theorem}
\providecommand*{\notatautorefname}{Notation}
\providecommand*{\defnautorefname}{Definition}

\providecommand*{\propositionautorefname}{Proposition}
\providecommand*{\exampleautorefname}{Example}

\theoremstyle{plain}
\newtheorem{lem}{\lemautorefname}[subsection]

\newaliascnt{theorem}{lem}
\newtheorem{theorem}[theorem]{\theoremautorefname}
\aliascntresetthe{theorem}

\newaliascnt{cor}{lem}
\newtheorem{cor}[cor]{\corautorefname}
\aliascntresetthe{cor}

\newtheorem{property}{\propertyautorefname}

\newaliascnt{proposition}{lem}
\newtheorem{proposition}[proposition]{\propositionautorefname}
\aliascntresetthe{proposition}

\newtheorem*{main}{Theorem \ref{mainthm}}

\newtheorem{ax}{\axautorefname}

\theoremstyle{definition}
\newaliascnt{defn}{lem}
\newtheorem{defn}[defn]{\defnautorefname}
\aliascntresetthe{defn}

\newaliascnt{example}{lem}

\aliascntresetthe{example}

\theoremstyle{remark}
\newaliascnt{remark}{lem}
\newtheorem{remark}[remark]{\remarkautorefname}
\aliascntresetthe{remark}

\newaliascnt{notat}{lem}
\newtheorem{notat}[notat]{\notatautorefname}
\aliascntresetthe{notat}

\author{Amanda Francis, Tyler Jarvis, Drew Johnson, and Rachel Suggs}
\title[Landau-Ginzburg Mirror Symmetry]{Landau-Ginzburg Mirror Symmetry for Orbifolded Frobenius Algebras}
\date{\today}

\begin{document}

\begin{abstract}
We prove the Landau-Ginzburg Mirror Symmetry Conjecture at the level of (orbifolded) Frobenius algebras for a large class of invertible singularities, including arbitrary sums of loops and Fermats with arbitrary symmetry groups.
Specifically, we show that for a quasi-homogeneous polynomial $W$ and an admissible group $G$ within the class,  the Frobenius algebra arising in the FJRW theory \cite{fjr1} of $[W/G]$ is isomorphic (as a Frobenius algebra) to the orbifolded Milnor ring \cite{krawitz, kau1, IV} of $[W^T/G^T]$, associated to the dual polynomial $W^T$ and dual group $G^T$.
\end{abstract}
\maketitle
\vspace*{-6ex}
\section{Introduction}
\subsection{Background}
The Landau-Ginzburg mirror symmetry conjecture says that the A-model of a given singularity with a  given symmetry group should match the B-model of an appropriately chosen  \emph{transpose singularity} and \emph{transpose group}.  The full conjecture is that there should be cohomological field theories for both the A- and B-models and that these should match in all genera.

The A-model has been constructed by Fan, Jarvis, and Ruan \cite{fjr2, fjr3, fjr1}, following ideas of Witten.  For any non-degenerate, quasi-homogenous (weighted homogeneous) polynomial $W$, and for any admissible group of diagonal symmetries $G$, they constructed a cohomological field theory which is often  called \emph{FJRW theory}.  A restriction of this theory to genus zero with three marked points gives a Frobenius algebra for the A-model which we denote $\HH_{W,G}$.

For the B-model, in the case that the orbifold group is trivial, the Frobenius algebra is given by the Milnor ring (or local algebra) of the singularity.  For cases orbifolded by non-trivial groups, the orbifold B-model, as a vector space, was given by Intriligator and Vafa \cite{IV}; however, until more recently the orbifold B-model was lacking a definition of the product structure.  In \cite{krawitz}, Krawitz, following ideas of Kaufmann \cite{ kau2, kau3, kau1}, wrote down a multiplication for the orbifold Milnor ring, which we call $\BB_{W,G}$.

The construction of the mirror dual of a given invertible polynomial was first described by Berglund-H\"ubsch \cite{ber-hub}.  It is defined for  so-called \emph{invertible singularities}, that is, polynomials with the same number of monomial as variables which define an isolated singularity at the origin.    The general construction of the dual group and the general definition of the mirror map (at the level of vector spaces) is due to Krawitz \cite{krawitz}, based in part on ideas from \cite{kreuzer}.

Krawitz proved that the LG Mirror Symmetry Conjecture holds on the level of bi-graded vector spaces for all invertible polynomials and all admissible groups.  At the level of Frobenius algebras, he only proved the conjecture for the case when the B-model has trivial orbifold group.  This generalized some previous work of \cite{priddis}, \cite{acosta}, and \cite{fan-shen}.  A few special cases of the conjecture with non-trivial orbifolded B-model have also been verified \cite{krawitz,  bergin}. 

This paper proves the LG Mirror Symmetry Conjecture for arbitrary admissible groups of symmetries.  We show that for a wide class of polynomials and admissible orbifold groups, the A-model is isomorphic, as a Frobenius algebra, to the orbifold B-model of the dual singularity and dual group.  The precise conditions on the polynomials and groups is given in \autoref{property}.  In terms of the classification \cite{KrSk} of invertible polynomials into Fermats, chains, and loops, our result applies many cases, including arbitrary sums of Fermats and loops, with arbitrary choice of admissible group.   

The product structure of the FJRW theory is determined by the genus-zero, three-point correlators, which are $\CC$-valued functions on $\HH_{W,G}^{\otimes 3}$.  The FJRW theory satisfies the axioms of a cohomological field theory as well as several additional axioms that facilitate computation.
 When orbifolding by a trivial symmetry group on the B-side, and thus by the maximal group on the A-side, the relevant insertions are mostly what we call \emph{narrow} \cite{fjr1}. In these cases, the axioms of \cite{fjr1} provide straightforward ways to compute the correlators.

When we orbifold by non-trivial symmetry groups on the B-side  (and thus smaller groups on the A-side) the situation is more difficult, more of the insertions may be \emph{broad}, and are much more difficult to compute.   The general case of the problem is a difficult PDE-problem, not yet explicitly solved in most cases.

Orbifolding by a smaller group on the A-side also introduces another difficulty.   When the orbifold group is a (decoupled) product of groups acting on these sums, as in the case of the maximal symmetry group, the FJRW ring is the tensor product of the pieces, and the B-model can similarly be broken up as a tensor product (see \autoref{ax:sums} and \autoref{prop:b-sums}).  Thus, previous attention has been focused on these atomic types.  However, if we consider more arbitrary groups are not necessarily decoupled products of groups, this method does not apply directly. This paper introduces a new strategy to take advantage of the ``breaking up into tensor products'' technique for more general orbifold groups.  This also allows us to avoid computation of some of the difficult correlators.  Essentially, for each product, we take a subalgebra containing the factors that can also be thought of as a subalgebra of a theory with a group that does break up as a direct product in a useful way.

Consider a non-degenerate, invertible singularity $W \in \CC[x_1,\dots,x_N] $ and an admissible group $G$ of symmetries of $W$.  We prove that LG mirror symmetry holds for all $(W,G)$ that have the following property:
\begin{property}\label{property}
Let $W$ be a non-degenerate, invertible singularity, and let $G$ be an admissible group of symmetries of $W$. We say that \emph{the pair $(W,G)$ has \autoref{property}} if, 
\begin{enumerate}
\item $W$ can be decomposed as
 $   W = \sum_{i=1}^{M} W_i ,$
where the $W_i$ are themselves invertible polynomials having no variables in common with any other $W_j$.  
\item For any element $g$ of $G$ whose associated sector  $\HH_g \subseteq \HH_{W,G}$ is nonempty, and for each $i\in \{1,\dots ,M\}$ the action of $g$ fixes either all of the variables in $W_i$ or none of them. 
\item For any element $g'$ of $G\tr$ whose associated sector of $\BB_{g'} \subseteq \BB_{W\tr,G\tr}$ is non-empty, and for each $i\in \{1,\dots ,M\}$ the action of $g'$ fixes either all of the variables in $W_i\tr$ or none of them.
\end{enumerate}
\end{property}
\begin{remark}
If $W$ is a two-variable chain, or if $W$ is a sum of Fermat and loop type singularities (see \autoref{atomictypes} of this paper),
then Property \ref{property} is satisfied for \emph{any} admissible group $G$, and \autoref{mainthm} will hold.
\end{remark}
Let $\HH_{W,G}$ and $\BB_{W\tr,G\tr}$ be the FJRW ring (LG A-model Frobenius algebra) and orbifold Milnor ring (LG B-model Frobenius algebra), respectively.
Our main theorem is the following.
\begin{main}
If $W$  and $G$ satisfy \autoref{property}, then there is an isomorphism of Frobenius algebras:
$    \HH_{W,G} \cong \BB_{W\tr,G\tr}$. 
\end{main}

Borisov \cite{borisov} has constructed both an A-model and a B-model Frobenius algebra, and he has proved mirror symmetry for his constructions in the case that the singularity is Calabi-Yau and the symmetry group is both admissible and a subgroup of $\SL_N$.  Borisov's B-model agrees with the Krawitz/Kaufmann $\BB_{W\tr,G\tr}$, but without our mirror symmetry result, it is not at all clear that Borisov's A-model agrees with the FJRW construction $\HH_{W,G}$.
Furthermore, our result applies to polynomials that are not necessarily Calabi-Yau, and to admissible groups that are not necessarily subgroups of $\SL_N$.

\section{Review of the Constructions} \label{sec:review}

\subsection{Quasihomogeneous Polynomials} \label{sec:qpoly}
We call a polynomial \emph{invertible} if it has the same number of variables as monomials.
We start with a quasi-homogeneous, invertible polynomial in variables $X_1, \dots X_N$:
\[
    W = \sum_{i=1}^N c_i\prod_{j=1}^N X_j^{a_{ij}} \in \CC[X_1, \dots, X_N].
\]
The matrix $\M= (a_{ij})$ encodes the exponents of the polynomial.  We require that the weights $q_i$ of $X_i$ be uniquely determined by the condition that $W$ have weighted degree $1$.   
We can write the $q_i$ as 
\begin{equation}\label{intweights}
q_i = \frac{w_i}{d} \quad\text{  with }\quad \gcd(d, w_1,\dots,w_N) = 1.
\end{equation}
In such a polynomial the $c_i$ can be absorbed by rescaling the variables, so we will always assume from now on that $c_i = 1$.

We also require that the polynomial define an isolated singularity at the origin; i.e., the system of equations $\left\{\frac{\partial W}{\partial X_i}=0\right\}$ has a unique solution at the origin.  A polynomial satisfying these two conditions we call \emph{non-degenerate}.  Invertible, non-degenerate, quasi-homogeneous polynomials are completely classified.  
\begin{proposition}[ \cite{KrSk}]\label{atomictypes}\label{sec:classification}
Any invertible, non-degenerate, quasi-homogeneous polynomial is the decoupled sum of polynomials of one of three \emph{atomic types}: 
\begin{align*}
  \text{Fermat \emph{type:} } & W = X^a,\notag\\
  \text{loop  \emph {type:} } & W = X_1^{a_1}X_2 + X_2^{a_2}X_3 + \dots + X_N^{a_N}X_1,\\
   \text{chain  \emph{type:} } & W = X_1^{a_1}X_2 + X_2^{a_2}X_3 + \dots + X_N^{a_N}.\notag
\end{align*}
\end{proposition}
We also assume that $a_i \ge 2$, so that there are no terms of the form $X_iX_j$.  

\begin{remark}
Assuming that the variables are listed in order as above, an atomic polynomial is determined by its type and by the vector of non-one exponents $\ba = 
(a_1, \dots, a_N)\tr$. 

Throughout this paper we will use boldface type to represent a column vector, and regular italic type to represent entries in the vector.  Thus, by $\bg$ we mean the vector $(g_1, \dots, g_N)\tr$, where the $N$ must be understood from context. We also write $\one$ for the vector $(1, \dots, 1)\tr$.
\end{remark}

We form the \emph{Berglund-H\"ubsch dual} or  \emph{transpose singularity} $W\tr$ by taking the polynomial with exponent matrix $\M\tr$.  
If $W$ is invertible and non-degenerate, then $W\tr$ is also invertible and non-degenerate.  We will generally denote the variables of $W\tr$ by $Y_i$.  Taking the transpose of any of the atomic types preserves the type.  The variables, however, will be in the reverse order.

\subsubsection{Diagonal symmetry groups}
There is an action of $(\CC^*)^N$ on $\CC[X_1, \dots, X_N]$, where the tuple $(\lambda_1, \dots, \lambda_N) \in (\CC^*)^N$ acts on $X_j$ by multiplication by $\lambda_j$.
\begin{defn}
The \emph{maximal group of diagonal symmetries} $\Gmax_W$, or simply $\Gmax$, if $W$ is clear from context, is the maximal subgroup of $(\CC^*)^N$ which fixes the polynomial $W$.
\end{defn}
We prefer to think of symmetries $\Gmax_W$ as a subgroup of $(\mathbb Q / \mathbb Z)^N$, where the element $[\bg]$ corresponding to the class of the vector $\bg \in \QQ^N$  acts on $X_j$ by multiplication by $\exp(2\pi i g_j)$.
We can find a special set of generators for $\Gmax_W$ and a special element $J$, as follows.
\begin{defn} \label{def:rho-and-J}
The group element in $(\QQ/\ZZ)^N$ corresponding to the $i$th column of $\M^{-1}$ we denote by $\rho_i$.

The element 
$
    J := \sum_{j=1}^N \rho_j = [\M^{-1}\one] = [\mathbf q] \in (\QQ/\ZZ)^N
$ 
 is called the \emph{exponential grading element}.
\end{defn}
\begin{proposition}[Krawitz]
 The $\rho_i$ generate the maximal symmetry group $\Gmax$. Similarly, the rows $\bar \rho_i$ of $\M^{-1}$ generate the maximal symmetry group of $W\tr$.
\end{proposition}

\begin{defn}[Krawitz]
The dual group $G\tr$ is a group of symmetries of $W\tr$ defined by
\begin{equation}
	G\tr := \left\{ [\bg] \mid \bg\tr \M \bb \in \ZZ^N \text{ for all } [\bb] \in G\right\}.
\end{equation}
\end{defn}
One can check that $G\tr$ is a group and that the definition is independent of the presentation of the elements of $G$.  Additionally, the transpose group has the following properties, which are verified in \cite{krawitz}:
\begin{itemize}
	\item If $G$ contains $J$, then $G\tr$ is contained in $\SL_N$, and vice-versa.
	\item $(G\tr)\tr = G$.
	\item If $G' \le G$, then $G\tr \le (G')\tr$. 
\end{itemize}

\subsubsection{Milnor Rings}
\begin{defn}
The \emph{Jacobian ideal} $\operatorname{Jac}(W)$ of a polynomial $W$ is the ideal generated by the partial derivatives:
\[
    \operatorname{Jac}(W) := \left(\frac{\partial W}{\partial X_1}, \dots, \frac{\partial W}{\partial X_N}\right)
\]
and the \emph{Milnor ring} $\QQQ_W$ (also called the \emph{local algebra}) is
\[
\QQQ_W := \mathbb C[X_1, \dots, X_N]/\operatorname{Jac}(W).
\]
\end{defn}
The Milnor ring is a finite-dimensional $\CC$-vector space, graded by the weighted degree of the monomials.  The subspace of highest weighted degree is one-dimensional, is spanned by the Hessian determinant $\hess(W)$, and has weighted degree equal to the \emph{central charge}  $\hat c  = \sum_{j=1}^N (1-2q_j)$.

The residue pairing $\langle\ ,\ \rangle$ makes $\QQQ_W$ into a Frobenius algebra; that is, for every $a,b,c \in \QQQ_W$ we have $$\langle a\cdot b , c \rangle = \langle a, b \cdot c \rangle.$$
The residue pairing in $\QQQ_W$ can be computed as $\left<a,b\right> = \mu \frac{ab}{\hess(W)}$, by which we mean 
\begin{equation}\label{respair}
ab = \mu\hess(W) \left<a,b\right>
 + \text{lower-degree terms},
 \end{equation}
   where $\mu$ is the dimension of the vector space $\QQQ_W$.

\subsection{The A-model}
The Landau-Ginzburg A-model Frobenius algebra is the so-called \emph{FJRW ring}.  We give a description in terms of Milnor rings which is more elementary than that used in \cite{fjr1} and is sufficient for our computations.
\begin{defn} Any subgroup of $G$ of $\Gmax$ containing the exponential grading element  $J$  is called \emph{admissible}.
For each $g \in \Gmax$, we define $W_g$ to be the restriction of the polynomial $W$ to the subspace of $\CC^N$ fixed by $g$.  
\medskip
\subsubsection{The A-model state space}
The state space of the FJRW ring is the underlying vector space of the ring.  This is constructed by taking a direct sum of \emph{sectors}---one for each $g\in G$.
\begin{defn}\label{defnarrow}
Let $W$ be a non-degenerate, invertible singularity and $G$ be admissible group of symmetries of $W$.  For any $g\in G$,  the \emph{unprojected $g$-sector} $\uHH_g$ is the $G$-module
\[
    \uHH_g := \QQQ_{W_g} \cdot dX_{i_1} \wedge \dots \wedge dX_{i_{N_g}},
\]
where $N_g$ is the number of variables fixed by $g$, and $i_1, \dots i_{N_g}$ are the indices of the fixed coordinates. 

The $g$-sector is called \emph{narrow} if $N_g = 0$, and \emph{broad} otherwise.    
Each $g = (\Theta_1,\dots,\Theta_N) \in G$ acts on both the variable $X_j$ and on the one-form $dX_j$ by multiplication by $\exp(2\pi i \Theta_j)$. 

The state space of the A-model, or underlying vector space of the FJRW ring, is given by taking the direct sum of the un-projected sectors and taking $G$ invariants:
\begin{equation}
    \HH_{W,G} := \left(\bigoplus_{g \in G} \uHH_g\right)^G 
    \dsand \HH_g := \uHH_g^G.
        \label{eq:HWG}
\end{equation}
We call $\HH_g$ the \emph{projected $g$-sector}, or \emph{$g$-sector of the FJRW ring.}
\end{defn}

We denote an element of $\uHH_g$ or $\HH_g$ by 
$\kn{m}{g},$
where $m$ is an element of the ring $\QQQ_{W_g}$.   We do not explicitly write the volume form, since it is determined by the group element.  This notation is not standard in the subject, but we find it clearer than other notations, especially when we are discussing the mirror map.
\end{defn}

\subsubsection{The pairing and three-point correlators}
The Milnor ring always has a basis of monomials, and since the $G$-action is diagonal, the $G$-invariants $\HH_g$ also have a basis of monomials.  Thus, we can write a basis for $\uHH_g$ or of $\HH_g$ with elements of the form $\kn{\X^\brr}{g}$, where  by $\X^{\brr}$ we mean the monomial $\prod_{j = 1}^{N_g} X_{i_j}^{r_j} \in \QQQ_g$.   

\begin{defn}
We endow $\HH_{W,G}$ with a pairing as follows.  Since $g$ and $-g$ fix the same coordinates, there is a natural isomorphism
$$
    I: (\HH_g)^G \rightarrow  (\HH_{-g})^G \quad \text{ given by } \quad
    \kn{\X^{\brr}}{g} \mapsto \kn{\X^{\brr}}{-g}.
$$
The pairing on the Milnor ring $\QQQ_g$ induces a pairing
\[
    \left< \cdot, \cdot \right>_g:(\HH_g)^G \otimes (\HH_{-g})^G \rightarrow  \mathbb C
\]
\[
    \left< \kn{\X^{\brr}}{g}, \kn{\X^{\bs}}{-g} \right>_g = \left< \kn{\X^{\brr}}{g}, I^{-1}(\kn{\X^{\bs}}{-g})\right>_{\QQQ_{W_g}}.
\]
The pairing on the full state space $\HH_{W,G}$ is the ``direct sum" of these pairings---that is, if two basis elements are from sectors $g$ and $-g$, we use the pairing $\left<\cdot, \cdot \right>_g$ above, and otherwise, the pairing is 0.   
\end{defn}

The FJRW theory provides a full cohomological field theory, but for the purposes of this paper, we are only concerned with the Frobenius algebra structure, arising from the genus-zero, three-point correlators, which we denote by 
\[
    \left< \cdot, \cdot, \cdot \right>^{W,G} : \HH_{W,G}^{\otimes 3} \rightarrow  \CC.
\]
We omit the superscripts ${W,G}$ when they are clear from context.  We discuss the computation of these correlators in \autoref{sec:axioms}.
The product in the FJRW ring is given by
\begin{equation}
    \alpha \star \beta = \sum_{\tau,\sigma} \left<\alpha, \beta, \tau\right> \eta^{\tau,\sigma} \sigma, \label{eq:mul}
\end{equation}
where the sum is over all pairs of elements $\sigma, \tau$ from a fixed basis of $\HH_{W,G}$, and where $\eta^{\tau,\sigma}$ is the inverse of the pairing matrix with respect to this basis. The FJRW ring $\HH_{W,G}$ is a graded Frobenius algebra with a $\QQ$-grading that we call the  \emph{$W$-degree}.

\subsubsection{A-model axioms} \label{sec:axioms}
The genus-zero, three-point correlators may be difficult to compute in general.  However, they satisfy some axioms that allow us to compute them  in many cases.  These axioms all follow immediately, by restriction to the genus-zero, three-point case, from the axioms for the FJRW virtual cycle described in \cite{fjr1}.  

Throughout this section, for $i \in\{ 1,2,3\}$ we assume that $\gamma_i \in \HH_{g_i} \subseteq \HH_{W, G}$ are elements of the FJRW ring, with $g_i = \bTheta^i = (\Theta_1^i,\dots,\Theta_N^i)\tr$ and  $0\le \Theta_j^i<1$ for every $i,j$.
We will only list here the axioms we need for this paper.  The remaining axioms can be found in \cite[Thm 4.1.8]{fjr1} and \cite[\S1.2]{priddis}.

\begin{ax}[Symmetry]
Let $\sigma \in S_3$.  Then
\[
    \left<\gamma_1, \gamma_2, \gamma_3\right> = \left<\gamma_{\sigma(1)}, \gamma_{\sigma(2)}, \gamma_{\sigma(3)}\right>.
\]
\end{ax}

The next axioms relate to the degree of certain line bundles.  For genus zero, these degrees are given by
\[
    l_j = q_j - \sum_{i=1}^3 \Theta^i_j \text{ for $j \in\{ 1,2,3\}$ }.
\]

\begin{ax}[Integer Line Bundle Degrees] \label{lbd-ax}
The correlator $\left<\gamma_1, \gamma_2, \gamma_3\right>$ vanishes unless $l_j \in \mathbb Z$ for $j=1,\dots, N$.
\end{ax}

The following observation, due to Krawitz, follows from \autoref{lbd-ax}.
\begin{proposition} \label{Jg-grading}
Suppose $\left<\gamma_1, \gamma_2, \gamma_3\right>$ does not vanish.  Then $g_3 = J - g_1- g_2$.
Thus $\gamma_1 \star \gamma_2 \in \HH_{g_1+ g_2 - J}$, and
$
    \HH_{g_1 + J} \star \HH_{g_2 + J} \subset \HH_{g_1 + g_2 + J}.
$
\end{proposition}
%


\begin{ax}[Pairing] \label{ax:pairing}
Let $\ringid = \kn 1J$.  Then
$\left<\gamma_1, \gamma_2, \ringid \right> = \pair{\gamma_1}{\gamma_2}.$
\end{ax}
This axiom implies that $\ringid$ is the identity element with respect to the multiplication \eqref{eq:mul} in the FJRW ring.  Do not confuse the identity element of the ring with the element $\kn{1}{0}$ in the identity sector.

%
\begin{ax}[Decoupled Sums] \label{ax:sums}
Suppose $W_1$ and $W_2$ are non-degenerate, quasi-homogeneous polynomials with no variables in common.  Suppose $G_1$ and $G_2$ are admissible groups of diagonal symmetries  for $W_1$ and $W_2$ respectively.  Then $G_1 \times  G_2$ is an admissible group of diagonal symmetries for $W_1 + W_2$.
Suppose
\[
    \kn{m_in_i}{g_i + h_i} \in \HH_{W_1 + W_2, G_1 \times  G_2}
\]
for $i = 1,2,3$, with $m_i \in \QQQ_{W_1}$, $n_i \in \QQQ_{W_2}$, $g_i \in G_1$, and $h_i \in G_2$.  Then the three-point correlator
\[
    \left< \kn{m_1 n_1}{g_1 + h_1}, \kn{m_2 n_2}{g_2 + h_2}, \kn{m_3 n_3}{g_3 + h_3} \right>^{W_1+W_2, G_1 \times  G_2}
\]
is equal to the product
\[
\left< \kn{m_1}{g_1}, \kn{m_2}{g_2}, \kn{m_3}{g_3} \right>^{W_1,G_1} \cdot \left< \kn{n_1}{h_1}, \kn{n_2}{h_2}, \kn{n_3}{h_3} \right>^{W_2,G_2}.
\]

This gives an isomorphism
\[
    \HH_{W_1, G_1} \otimes \HH_{W_2, G_2} \cong \HH_{W_1 + W_2, G_1 \times  G_2}.
\]
\end{ax}

\begin{ax}
\label{gmax-inv}
The three-point correlator is invariant under the action of $\Gmax$, i.e.,
\[
    \left<h\gamma_1, h\gamma_2, h\gamma_3\right> = \left<\gamma_1, \gamma_2, \gamma_3\right>
\]
for all $h \in \Gmax$.
\end{ax}

\begin{remark} \label{rem:gmax-inv}
\autoref{gmax-inv} gives another valuable selection rule as follows.  
 If $\gamma_i = \kn{\X^{\brr_i}}{g_i}$ with $g_i \in G_W$ for $i\in \{1,2,3\}$, then for  every $h = [\bh] \in \Gmax_W$, denote by $\bhbar{i}$ the restriction of $\bh$ to the fixed locus of $g_i$; that is, the $j$th coordinate of $\bh$ is set to zero in $\bhbar{i}$ if the $j$th coordinate of $g_i$ is non-trivial. The correlator  $\left<\kn{\X^{\brr_1}}{g_1}, \kn{\X^{\brr_2}}{g_2},\kn{\X^{\brr_3}}{g_3}\right>^{W,G}$ vanishes unless
\begin{equation}\label{newselrule}
    \sum_{i=1}^3 (\brr_i + \one)\tr \bhbar{i} \in \ZZ.
\end{equation}
\end{remark}

\subsection{The B-model}
The Frobenius algebra for the B-model we call the \emph{orbifolded Milnor ring}.
In this paper, on the B-side we will only consider diagonal automorphism groups  which are subgroups of $\SL_N$.  Here, we are thinking of the elements of $G$ as linear transformations of $\CC^N$.  In our notation, $ g  = [\bTheta]  \in \SL_N$ is equivalent to $\sum_i \Theta_i \in \ZZ$.  As mentioned above, this condition is dual to the condition that the A-model orbifold group contain $J$.

The construction of the state space (underlying vector space) is identical to the construction of the FJRW state space.  We again take the un-projected $g$ sector
\[
    \uBB_g = \QQQ_{W_g} \cdot dX_{i_1} \wedge \dots \wedge dX_{i_{N_g}}
\]
with the same $G$-action as before, and again let the (projected) $g$-sector $\BB_g$ of our B-model be the $G$-invariants:
\[
\BB_g := \uBB_g^G \quad \text{ and } \quad    \BB_{W,G} = \left(\bigoplus_g \uBB_g\right)^G.
\]
The pairing is also defined in the same way as for the A-model.  

The main difference between the A- and B-models is in the product.  The B-model product makes $\BB_{W,G}$ into a $G$-graded Frobenius algebra.  In fact the product does more---it can be defined on the entire unprojected state space $\uBB_{W,G} := \bigoplus_{g\in G} \uBB_g$ and there it defines a \emph{$G$-Frobenius algebra} (cf. \cite{kau1, JKK:05, Tur:10}).
 
The product was first written down explicitly in \cite{krawitz}, following ideas from \cite{kau1}.  To define the product on $\BB_{W,G}$, first note that for all $g$ there is a surjective homomorphism of Milnor rings $\QQQ_e \rightarrow  \QQQ_g$ given by setting to zero all variables not fixed by $g$.  Thus, $\QQQ_g$ may be thought of as a cyclic $\QQQ_e$-module with generator $\kn{1}{g}$.   For any two elements $g,h\in G$ we define $W_{g,  h}$ to be the restriction of the polynomial $W$ to the subspace of $\CC^N$ fixed by both $g$ and $h$.  Define 
$$\QQQ_{g,h}:= \QQQ_{W_{g,  h}}$$
to be the Milnor ring of $W_{g,  h}$.  There are three surjective ring homomorphisms:
$${e_g}: \QQQ_g \rightarrow  \QQQ_{g,h}, \qquad  {e_h}:\QQQ_h 	\rightarrow   \QQQ_{g, h} \dsand    {e_{g+h}}:  \QQQ_{g+h} \rightarrow  \QQQ_{g,h}.
$$
The residue pairing $\eta$ is non-degenerate for each of the Milnor rings, and hence we can use it to construct vector-space isomorphisms $\eta_{g,h}^{\flat}:\QQQ_{g,h} \rightarrow  \QQQ_{g,h}^{\vee}$ and $\eta_{g+h}^{\sharp}:\QQQ_{g+h}^{\vee} \rightarrow  \QQQ_{g+h}$, where $^{\vee}$ indicates the dual vector space.   We can also can dualize the homomorphism $e_{g+h}$ to construct an injective linear map $e_{g+h}^{\vee}:   \QQQ_{g, h}^{\vee} \rightarrow   \QQQ_{g+h}^{\vee}$.  Combining these we  define 
$$e_{g+h}^* = \eta^{\sharp}_{g+h} \circ e_{g+h}^{\vee} \circ \eta^{\flat}_{g,h}: \QQQ_{g,h} \rightarrow  \QQQ_{g+h}.$$
Finally, we define $I_g$, $I_h$, $I_{g+h}$ to be the sets of indices  fixed by the group elements $g$, $h$, and $g+h$, respectively.

\begin{defn}\label{Bprod}
If $\kn{m}{g} \in \BB_g$ and $\kn{n}{h}\in \BB_h$, then $m\in\QQQ_g$ and $n\in \QQQ_h$.  We define the product  $\kn{m}{g} \bstar \kn{n}{h}$ to be 
\begin{equation}\label{eq:bmul}
\kn{m}{g} \bstar \kn{n}{h} := \begin{cases}
\kn{e_{g+h}^*\left(e_g(m) \cdot e_h(n)\right)}{g+h} & \text{ if  $I_g \cup I_h \cup I_{g+h} = \{1,2,\dots, N\}$}\\
0 & \text{ otherwise,}
\end{cases}
\end{equation}
and extend the product linearly to the rest of $\BB_{W,G}$.
\end{defn}

Using Equation~\eqref{respair}, a straightforward computation shows that \autoref{Bprod} is equivalent to the formulation given in \cite{krawitz}:
\[
    \mu_{g+h} \hess(W_{g, h}) \kn{1}{g} \bstar \kn{1}{h} = \mu_{g,h} \hess(W_{g+h}) \kn{1}{g+h},
\]
where $\hess(W_{g, h})$ and $\hess(W_{g+h}) $ denote the Hessian determinants of $W_{g,h}$ and $W_{g+h}$, respectively,  and   $\mu_{g, h}$ and $\mu_{g+h}$ denote the dimensions of the vector spaces $\QQQ_{g,h}$ and $\QQQ_{g+h}$, respectively. \cite{krawitz} shows this product is associative for all invertible $W$ and all $G\le \SL_N$.  

The orbifold Milnor ring also has a $\QQ$-grading that matches the $W$-degree in our mirror symmetry. 

The orbifold Milnor ring also has a tensor product property analogous to \autoref{ax:sums}.  This is easy to check from the definitions. 
\begin{proposition} \label{prop:b-sums}
Suppose $W_1$ and $W_2$ are non-degenerate, quasi-homogeneous polynomials with no variables in common.  Suppose $G_1\le \SL_{N_1}$ and $G_2\le \SL_{N_2}$ are groups of diagonal symmetries. Then $G_1 \times  G_2$ is contained in $\SL_{N_1+N_2}$ and is a group of diagonal symmetries for $W_1 + W_2$. Moreover, there is an isomorphism
\[
    \BB_{W_1, G_1} \otimes \BB_{W_2, G_2} \cong \BB_{W_1 + W_2, G_1 \times G_2}.
\]
\end{proposition}

\section{The mirror map}

Consider the case of a non-degenerate, invertible, quasi-homogeneous polynomial $W$ and an admissible group $G$ of symmetries of $W$.
Our main result is 
\begin{theorem} \label{mainthm}
If $(W,G)$ 
satisfies \autoref{property}, then there is an isomorphism of Frobenius algebras:
$	\HH_{W,G} \cong \BB_{W\tr,G\tr} $.
\end{theorem}

Marc Krawitz defined two mirror maps in \cite{krawitz}.  The first is an isomorphism of graded vector spaces $\BB_{W\tr,G\tr} \rightarrow  \HH_{W,G}$, defined for any $(W,G)$, but it is not necessarily a ring homomorphism.  The second is an algebra isomorphism $\BB_{W\tr,0} \rightarrow  \HH_{W,\Gmax}$, but it is not defined for general groups.  In this section, we will show how to combine these two maps to get a mirror map defined for all pairs $(W,G)$ satisfying \autoref{property}.

\subsection{Basic Setup}\label{BasicSetup}

The fact that $(W,G)$ satisfies \autoref{property} means that $W$ has a decomposition   
\begin{equation}
	W = \sum_{i=1}^M W_i \label{eq:Wsum},
\end{equation}
where each $W_i$ is itself a non-degenerate, invertible polynomial,  having no variables in common with any of the other $W_j$, and such that \autoref{property} holds for this decomposition.  Unless otherwise stated, we will assume hereafter that this decomposition has been fixed, once and for all.

\begin{notat} \label{not:split}
Notice that $\HH_{W,G}$ has a basis of element of the form $\kn{m}{g}$ with  $m = \prod_{i=1}^M m_i$, where $m_i$ is a monomial in $\QQQ_{W_i}$.  For any $g\in G$ we can also write $g = \dseo{g_i}{i}$, where $g_i$ acts trivially except on $W_i$.  Hereafter, when  any element of $\HH_{W,G}$ is written in this form, it will mean that the element is as described above.  Similar remarks apply to the B-model.
\end{notat}

\subsubsection{The maximal unprojected state space}
For any quasi-homogeneous, non-degenerate, invertible polynomial $W$ and any  group $H\le\SL_N$ of diagonal symmetries of $W$, the projected state space $\BB_{W\tr,H}$ is a subspace of the unprojected vector space $\uBB_{W\tr,\Gmax_{W\tr}}$.  Moreover, since the definition of B-model multiplication makes no reference to the orbifold group, but only to the group elements involved, the multiplication makes sense on all of $\uBB_{W\tr,\Gmax_{W\tr}\cap \SL_N}$.  This gives us the following lemma. 
\begin{lem}\label{lem:b-change-group}
 For any diagonal symmetry group $H \le \SL_N$  of $W\tr$, the subspace $\BB_{W\tr, H}$ is  a subalgebra of $\uBB_{W\tr,\Gmax_{W\tr}\cap \SL_N}$.  

If $B$ is a subalgebra of $\BB_{W\tr, H}$ that is invariant under another symmetry group $H'\le \SL_N$, and if $B$ contains only sectors with group elements in $H'$, then $B$ is also a subalgebra of $\BB_{W\tr, H'}$.
\end{lem}

\begin{remark}\label{Amasterstatespace}
On the A-side, we have a similar situation at the level of vector spaces: for any non-degnerate $W$ and admissible $G$,  the projected state space $\HH_{W,G}$ is a subspace of the unprojected vector space $\uHH_{W,\Gmax}$.  There is, however, no obvious algebra structure on the unprojected space $\uHH_{W,\Gmax}$ that would project to give the algebra structure on every $\HH_{W,G}$.   
\end{remark}

\subsubsection{Properties of Atomic Types}

Assume that $W$ is one of the atomic types with vector of exponents $\ba$, as in \autoref{atomictypes}.  To define the mirror map we need to define the \emph{Chain Property}. 
\begin{property}\label{chainproperty}
If $W$
is a chain 
we say that a vector $\bbeta$ has the \emph{Chain Property for $W$} if $\bbeta = (a_1-1, 0 , a_3-1, 0, \dots, a_{2k-1}-1, 0, \beta_{2k+1}, \beta_{2k+2},\dots,\beta_N)$, where $\beta_{2k+1} \neq a_{2k+1}-1$,
 and $\beta_{i} \leq a_{i}-1$ for all $i>2k$.  In other words, if $s\in \{1,2,\dots,N\}$ is the smallest index for which
$\beta_s \neq \delta_\text{odd}^s (a_s -1)$, then $s$ is odd.
\end{property}

The following lemma from \cite{kreuzer} gives a description of the Milnor ring of an invertible, non-degenerate polynomial.
\begin{lem}\label{lem:milnor}
\mbox{}
\begin{enumerate}
    \item The Milnor ring of a loop type polynomial is generated over $\CC$ by the basis 
    $\{\prod_{i=1}^N X_i^{\beta_i} : 0 \le \beta_i \le a_i -1\}$ and has dimension 
    $\mu_W=\prod_{i=1}^N a_i$. 
    \item\label{lem:chainmilnor}  The Milnor ring of a chain type polynomial is generated over $\CC$ by the basis
    $ \{\prod_{i = 1}^n X_i^{\beta_i} : 0 \le \beta_i \le a_i-1\},$ where the vector $\bbeta$ has the Chain \autoref{chainproperty}.
    This Milnor ring has dimension $\mu_{W} = \sum_{i = odd}(a_i - 1) \prod_{j  =i+1}^n a_j$.
    \item The Milnor ring of a Fermat type polynomial is generated over $\CC$ by the basis $\{X^{\beta} : 0 \le \beta \le a_i -2\}$ and has dimension $\mu_W=a-1$.
\end{enumerate}
In each case, define $\expon = \M\tr\one - \one$, then the monomial $\X^{\expon-\one}$ is of top degree, so the Hessian is a scalar multiple of $\X^{\expon-\one}$.
\end{lem}

To define the mirror map, we will also need the following lemma which follows trivially from \cite{kreuzer}.
\begin{lem}\label{lem:write-group} \label{lem:write-sym-groups}
    Any non-trivial diagonal symmetry of an atomic invertible singularity $W$ can be written uniquely in the form 
\begin{equation}\label{niceform}
          J+\sum_{i=1}^N \rr_i \rho_i, \quad\text{ or equivalently } \quad \M_W^{-1}(\brr + \one),
\end{equation}
where the vector $\brr$ satisfies the following constraints:
\begin{enumerate}
 \item\label{item:write-loop-group} 
    If $W$ is a loop polynomial, then $0 \le \rr_i \le a_i -1$.
  Furthermore,  if $N$ is even, then the identity element can be written as 
    $$
        0 =  J+\sum_{i \text{ even}} ({a_i-1})\rho_i= J+\sum_{i \text{ odd}} ({a_i-1})\rho_i.
    $$
    (If $N$ is odd, then the identity cannot be written in this form.)
    
    \item\label{item:write-chain-group}  If $W$ is a chain polynomial, then 
        $\brr$ satisfies the Chain \autoref{chainproperty}. The identity element can be written
    $$
        0 = J + \sum_{n-i \text{ even}} ({a_i-1})\rho_i.
    $$
    \item\label{item:write-fermat-group}  If  $W$ is a Fermat polynomial, then 
    $N=1$ and $\brr = (\rr)$ with $0\le \rr \le a-2$.
\end{enumerate}
\end{lem}
Hereafter, all elements of the Milnor rings and  groups
 associated to Fermat, loop, and chain type polynomials will be  written in the  forms described above.

\subsection{The mirror map}
In this section we will define the mirror map.  We do this first on some basic pieces we call \emph{fundamental factors} and then use these to define it in general.

\subsubsection{Fundamental factors}\label{fundamentalmmap}

If a polynomial is atomic, it cannot be written as a disconnected sum of other polynomials, and \autoref{property} implies a stricter property, which we call \emph{fundamental}.
\begin{defn}\label{def:fund}
We say that a pair $(W,G)$ is \emph{fundamental} if 
\begin{enumerate}
\item For every non-trivial $g\in G\tr$ the sector $\BB_g$ of $\BB_{W\tr,G\tr}$ is either narrow  (see \autoref{defnarrow}) or is empty. 
\item For every non-trivial $h\in G$ the sector  $\HH_h$ of $\HH_{W,G}$ is either narrow or is empty.
\end{enumerate}
\end{defn}
We construct the mirror map $\mmap:\BB_{W\tr,G\tr} \rightarrow  \HH_{W,G}$ for all fundamental pairs $(W,G)$ in two parts, as follows.  

Firstly, the untwisted sector $\BB_0 \subseteq \BB_{W\tr,G\tr}$ is a subalgebra of the untwisted sector of the unprojected algebra $\uBB_{0} \cong \QQQ_{W\tr}$,  and Krawitz's unorbifolded algebra mirror map \cite[\S3.1]{krawitz} gives an isomorphism of this sector to the sum of the A-model twisted sectors (all narrow), $\bigoplus_{h\in \Gmax_W} \HH_{h} \subseteq \HH_{W,\Gmax_W}$.   Our mirror map is a rescaling of that isomorphism, given on the  algebra generators $\kn{Y_j}{0}$ by
\begin{equation}
	\kn{Y_{j}}{0} \mapsto \rcb^{\wb_j} 
	\kn{1}{\rho_{j} + J} \label{eq:mymm2},
\end{equation}
where $\wb_j$ is the integer weight of $Y_j$ as defined in Equation~\eqref{intweights}, and where $\rcb\in \CC$ is a constant chosen such that  
\begin{equation}\label{rescaledef-b}
\rcb^{\bwb \cdot \expontr }= \left<\kn{\Y^{\expontr}}{0},\kn{1}{0}\right> = \frac{\mu_{W\tr} \Y^{\expontr}}{\hess(W\tr)}
\end{equation}
Also, $\mmap$ takes the identity  $\kn{1}{0} \in \BB_{W\tr,G\tr}$ to the identity  $\ringid =\kn{1}{J}\in\HH_{W,G}.$ 

Secondly, when $g\in G\tr$ is non-trivial, the sector must be narrow or empty.  If we write $W$ as a decoupled sum of atomic types, then by \autoref{lem:write-sym-groups} we may write $g$ uniquely as  $g = (\M\tr)^{-1}(\brr + \one)$, with the components of $\brr$ satisfying the constraints of that lemma, as determined by the corresponding atomic types.  The direct sum of all the twisted sectors $\bigoplus_{g\neq 0 \in G\tr} \BB_g$ is spanned by elements of the form $\kn{1}{g}$.
We define the mirror map by 
\begin{equation}
	\kn{1}{(\M_W\tr)^{-1}(\brr + \one)} \mapsto  
	\rca^{-\bwa\cdot\brr} 
	\kn{\X^{\brr}}{0}, \label{eq:mymm1}
\end{equation}
where $\wa_j$ is the integer weight of $X_j$ as defined in Equation~\eqref{intweights}, and where $\rca\in \CC$ is a constant chosen such that  
\begin{equation}\label{rescaledef-a}
\rca^{\bwa \cdot \expon }= \left<\kn{\X^{\expon}}{0},\kn{1}{0}\right> = \frac{\mu_{W} \X^{\expon}}{\hess(W)}
\end{equation}
\begin{remark} \label{rem:agrees}
The mirror map on the twisted sectors is a rescaling of Krawitz's vector space mirror map. The lemmas in \cite{krawitz} show that the map \eqref{eq:mymm2} on the untwisted sector agrees with the vector space mirror map, except possibly in the case of an even variable loop.  In this case, there is some ambiguity in the definition of the vector space mirror map on a certain two dimensional subspace.  Krawitz made an arbitrary choice to prove the vector space isomorphism, and the algebra mirror map \eqref{eq:mymm2} should, in principle, resolve this ambiguity.  However, \eqref{eq:mymm2} is given in terms of algebra generators, and we have not been able to compute the image of the vector space basis elements in question.  It is possible to check, however, that the map is surjective.  Our map agrees with Krawitz's vector space map, up to a possibly different choice on the ambiguous subspace, so it is still a bijection.  Thus, it only remains to check that $\mmap$ is a ring homomorphism, which we do in   \autoref{sec:ainc}.
\end{remark} 
 
\subsubsection{Decoupled sums}

Let $W = \sum_{i=1}^M W_i$ be a decoupled sum of non-degenerate, invertible polynomials, and let $G = \bigoplus_{i=1}^M G_i$, where each $G_i$ is an admissible group of symmetries of $W_i$.  If the mirror map is defined for each of the pairs $(W_i,G_i)$ 
$$\mmap_i :\BB_{W_i\tr, G_i\tr} \rightarrow  \HH_{W_i, G_i},$$
then the mirror map for the decoupled sum is just given by the tensor product:
$$\mmap : \BB_{W,G} \cong \bigotimes_i \BB_{W_i\tr, G_i\tr} \stackrel{\otimes_i \mmap_i }{\longrightarrow} \bigotimes_i \HH_{W_i, G_i}  \cong \HH_{W,G}.$$

\subsubsection{The general case} \label{sec:new-mm}
We now wish to define the mirror map $\mmap$ for the general case of $(W,G)$ satisfying \autoref{property}, where $G$ is not necessarily a direct sum of groups.

\begin{notat} \label{not:g_I}
Let $I$ a subset of $\{1, \dots, M\}$.   Then we let 
\begin{equation}\label{subInota}
	 g_I = \dseo{g_i}{i \in I} \dsand
	 m_I = \prod_{i \in I} m_i \dsand
	 W_I\tr = \sum_{i \in I} W_i\tr.
\end{equation}
\end{notat}

 The algebra $\BB_{W\tr,G\tr}$ is spanned by elements of the form $\kn{\prod m_i}{\dse {g_i}}$.  For each such element, we will find a partition $\II$ of $\{1,2,\dots,M\}$ and for each $I\in\II$ we will find groups $G_I$ of symmetries of $W_I$ so that each pair $(W_I,G_I)$ is fundamental and $\kn{\prod m_i}{\dse {g_i}} \in \bigoplus_{I\in\II}\BB_{W_I\tr,G_I\tr}$ (and which satisfy a few other properties).  The previous results allow us to define the mirror map on the decoupled sum $\bigoplus_{I\in\II}\BB_{W_I\tr,G_I\tr}$, and we can show that the image actually lies in $\HH_{W,G}$.  We will show that the mirror map thus defined is independent of the choices and is a homomorphism of Frobenius algebras.

\begin{defn} \label{def:nicely}
Let $\II$ be a partition of $\{1,2,\dots,M\}$ and let $\{G_I \}_{I\in\II}$ be admissible groups of symmetries of $W_I$. 

For any $\kn{\prod m_i}{\dse{ g_i}}\in \BB_{W\tr,G\tr}$ we say that  \emph{$\kn{\prod m_i}{\dse{ g_i}}$ splits nicely} with respect to $\II$ and $\{G_I\tr\}$ if the following properties are satisfied:
\begin{enumerate}
	\item For every $I\in \II$ the pair $(W_I, G_I)$ is fundamental.
	\item $\kn{m_I}{g_I} \in \BB_{W_I\tr,G_I\tr}$. \label{nice:inBB}
	\item The element $g_I$ is either trivial, or acts non-trivially on all $W_i$ for $i \in I$. \label{nice:triv-or-not}
	\item If $g_I = 0$ and $m_I \neq 1$, then $|I| = 1$, i.e., $W_I$ = $W_j$ for some $j$. \label{nice:split-id}
\end{enumerate} 
\end{defn}

\begin{lem}\label{allsplit}
For any element $\gamma = \kn{\prod m_i}{\dse {g_i}} \in \BB_{W\tr,G\tr}$ there exists a partition $\II$ and groups $\{G_I\}$ so that $\gamma$ splits nicely with respect to $\II, \{G_I\}$.
\end{lem}
\begin{proof}
Let $I_0$ be the set of $i$ so that $g_i = 0$.  Let $I_g$ be the set of $i$ with $g_i \neq 0$. Then we choose the partition
\[
	\II = \{I_g\} \cup \{\{i\}\}_{i \in I_0}.
\]  
We choose the group $G_{I_g}\tr$ for $W_{I_g}\tr$ to be the group generated by $g_{I_g}$, and we choose the trivial groups for the others.  In light of \autoref{property}, it is clear that these groups satisfy the properties described in \autoref{def:nicely}. 
\end{proof}

For each $\II,\{G_I\}$ with each $(W_I,G_I)$ fundamental, define the mirror map $$\mmap_\II:\BB_{\sum_{I\in \II} W\tr_I,\bigoplus_{I\in II} G\tr_I} \cong \bigotimes_{I\in\II} \BB_{W\tr_I,G\tr_I} \rightarrow  \bigotimes_{I\in\II} \HH_{W\tr_I,G\tr_I} \cong  \HH_{\sum_{I\in\II}W_I,\bigoplus_{I\in\II}G_I}.$$
Given a $\II, \{G_I\}$ such that $\gamma$ splits nicely,  define the sub-algebra 
$$B_\II :=\BB_{W\tr,G\tr}\cap \BB_{\sum_{I\in\II}W\tr_I,\bigoplus_{I\in\II}G\tr_I}$$
and  apply the map $\mmap_\II$ to $B_\II$.   We observe from the definition of $\mmap$ that the image  $\mmap_\II(B_\II)$ lies in the vector space $\HH_{W,G}$. 

\autoref{cor:change-group-product} 
states that the algebra structure induced on the vector space $\HH_{W,G} \cap \HH_{\sum_{I\in\II}W_I,\bigoplus_{I\in\II}G_I}$ from $\HH_{\sum_{I\in\II}W_I,\bigoplus_{I\in\II}G_I}$ is the same as the algebra structure induced from $\HH_{W,G}$.  Thus the subalgebra $\mmap_\II(B_\II)$ of $\HH_{W, \bigoplus_{I \in \II} G_I}$
is a subalgebra of $\HH_{W,G}$.  Furthermore, we have the  following lemma.
\begin{lem}
The image $\mmap_\II(\gamma)$ is independent of the choice of $\II$ and $\{G_I\}$.
\end{lem}
This follows from the fact that Equation \eqref{eq:mymm1} respects the splitting, and both \eqref{eq:mymm1} and \eqref{eq:mymm2} refer only to the group elements, and not to the subgroup.

\begin{remark} \label{sec:g-v-s}
By \autoref{rem:agrees}  our map for fundamental pairs agrees with Krawitz's vector space isomorphism (up to a minor ambiguity for even-variable loops).  Since these maps all respect splitting of decoupled sums,  our $\mmap$ also preserves the grading and is a vector space isomorphism.  
\end{remark}				
Combining all these results we have the following.
\begin{proposition}\label{mirrorwelldef}
The map $\mmap:\BB_{W\tr,G\tr}\rightarrow \HH_{W,G}$ is a well-defined, degree-preserving isomorphism of vector spaces.  
\end{proposition}

\section{Computations in $\HH_{W,G}$} \label{sec:ainc}
In this section we prove two results about computations of products or correlators in $\HH_{W,G}$.  The first shows that we can change groups without changing the products and the second gives a selection rule that allows us to prove that many correlators vanish.
 
\subsection{Changing groups}
As mentioned in \autoref{Amasterstatespace}, for any admissible group $G$ of symmetries of $W$, the projected state space $\HH_{W,G}$ is a vector subspace of the maximal unprojected state space $\uHH_{W,\Gmax}$, but we have no algebra structure on $\uHH_{W,\Gmax}$.  In this section we address this problem with the following proposition.

\begin{proposition} \label{cor:change-group-product}
If $H$ and $H'$ are two admissible groups of symmetries of $W$, then  for any $\kn{m}{g}$ and $\kn{n}{h}$  in $\HH_{P,H} \cap \HH_{P, H'}$, the product $\kn{m}{g} \star \kn{n}{h}$ is the same whether is it computed in $\HH_{P,H}$ or $\HH_{P, H'}$.
\end{proposition}

The key observation needed to prove this proposition is the following lemma.
\begin{lem} \label{lem:change-group}
	Let $P$ be a quasi-homogeneous, non-degenerate, invertible polynomial and $H$ and $H'$ be admissible subgroups of $\Gmax_{P}$.  Suppose we have $\kn{m}{g}$, $\kn{n}{h}$, and $\kn{p}{k}$ in the vector space $\HH_{P,H} \cap \HH_{P,H'}$.  We have
\begin{equation}
	\left< \kn{m}{g}, \kn{n}{h}, \kn{p}{k} \right>^{P,H} = \left< \kn{m}{g}, \kn{n}{h}, \kn{p}{k} \right>^{P,H'}.\label{correlator}
\end{equation}
That is, we may compute the three point correlator in either FJRW ring.
\end{lem}
\begin{proof}
It suffices to prove the result in the case that  $H \leq H'$.  The correlator $\left<\kn{m}{g}, \kn{n}{h}, \kn{p}{k}   \right>^{P,H}$
is defined (see \cite[Def 4.6.2]{fjr1}) as 
\[
	\left<\kn{m}{g}, \kn{n}{h}, \kn{p}{k}   \right>^{P,H}:=\int_{\MM_{0,3}} \Lambda^{P,H}_{0,3}(m, n, p),
\]
where ${\MM_{0,3}}$ is the  stack of three-pointed, genus-zero stable curves, and $\Lambda^{P,H}_{0,3}(m,n,p)$ is defined (see \cite[Def 4.2.1]{fjr1}) to be the Poincar\'e dual of the pushforward of the virtual cycle, capped with the classes $m$, $n$, and $p$:
\begin{equation*}
\Lambda^{P,H}_{0,3}(m, n, p):= \frac{1}{\deg(\st^{P,H})}PD \,\st^{P,H}_*\left(\left[\W_{0,3,H}(P;{g},{h},{k}) \right]^{vir} \cap  
(m \cup n \cup p)
\right).
\end{equation*}
Here $\st^{P,H}:\W_{0,3,H}(P;{g},{h},{k})\rightarrow  \MM_{0,3} $ is the forgetful map taking the stack of genus-zero, stable $P$-curves with admissible group $H$ to the  stack of genus-zero stable curves with three marked points, defined simply by forgetting the $P$-structure on the curve.  The correlator   $\left<\kn{m}{g}, \kn{n}{h}, \kn{p}{k}   \right>^{P,H'}$
is defined similarly.

According to \cite[2.3.1]{fjr1}, there is a finite morphism of stacks $\adm:\W_{0,3,H}(P;{g},{h},{k})\rightarrow 
\W_{0,3,H'}(P;{g},{h},{k})$, surjective onto an open and closed substack of $\W_{0,3,H'}(P)$.  Moreover, $\W_{0,3,H'}(P)$ actually has only a single geometric point, corresponding to the unique genus-zero, three-pointed $P$-curve with markings ${g}$, ${h}$, and ${k}$, respectively.  Therefore, in this case, the morphism $\adm$ is surjective and finite.

Theorem 6.3.5 of \cite{fjr3} shows that the virtual class $\left[\W_{0,3,H'}(P;{g},{h},{k}) \right]^{vir} $ is the pullback along $\adm$ of the virtual cycle $\left[\W_{0,3,H'}(P;{g},{h},{k}) \right]^{vir}$ on $\W_{0,3,H'}(P)$, which gives
\begin{align*}
  \left<\kn{m}{g},\kn{n}{h}, \kn{p}{k}   \right>^{P,H} 
&=\int_{\MM_{0,3}} \frac{PD \,\st^{P,H'}_*\adm_*\left(\adm^*\left[\W_{0,3,H'}(P;{g},{h},{k}) \right]^{vir} \cap  
(m \cup n \cup p)
\right)}{\deg(\st^{P,H})}\\
&= \int_{\MM_{0,3}} \frac{\deg(\adm) PD \,\st^{P,H'}_*\left(\left[\W_{0,3,H'}(P;{g},{h},{k}) \right]^{vir} \cap  
(m \cup n \cup p)
\right)}{\deg(\st^{P,H})}\\ 
 &=\left<\kn{m}{g}, \kn{n}{h}, \kn{p}{k}   \right>^{P,H'}.
\end{align*}\end{proof}
\begin{proof}[Proof of \autoref{cor:change-group-product}]
By definition of multiplication, we have
\[
	\kn{m}{g} \star_{P,H} \kn{n}{h} = \sum_{\sigma, \tau} \left<\kn{m}{g}, \kn{n}{h}, \sigma\right> \eta^{\sigma, \tau} \tau,
\]
where $\sigma$ and $\tau$ range over a basis of $\HH_{P,H}$.  The product is defined similarly for $\HH_{P,H'}$.  For elements $\HH_{P,H} \cap \HH_{P,H'}$, \autoref{lem:change-group} tells us that we can compute the correlators (and thus also the pairing) in either.  It suffices then to show that if we have a basis element $\kn{p}{k}$ of $\HH_{P,H}$ that is not in $\HH_{P,H'}$, the correlator
$	\left< \kn{m}{g}, \kn{n}{h}, \kn{p}{k} \right>^{P, H} $
vanishes. 
This is a straightforward consequence of  
\autoref{gmax-inv} ($\Gmax$-invariance) and \autoref{Jg-grading}. \end{proof}

\subsection{A selection rule}
We isolate the following fact since it will be used multiple times in the proof of \autoref{lem:vanish-as-well}.
\begin{proposition} \label{lem:break-off}
Suppose $P$ is a decoupled sum $P_1 + P_2$, and following \autoref{not:split}, we have three basis elements $\kn{m_1\cdot m_2}{g_1 + g_2}$, $\kn{n_1\cdot n_2}{ h_1 + h_2}$, and $\kn{p_1\cdot p_2}{ k_1 + k_2}$ of $\HH_{P,H}$.  Suppose $\kn{m_1}{g_1}$ and $\kn{n_1}{h_1}$ are invariant under $\Gmax_{P_1}$.  Then the correlator
\begin{equation}
 \left< \kn{m_1\cdot m_2}{g_1 + g_2}, \kn{n_1\cdot n_2}{ h_1 + h_2}, \kn{p_1\cdot p_2}{ k_1 + k_2}\right> \label{eq:breakable-corr}
\end{equation}
vanishes unless both
\begin{enumerate}
\item $\kn{p_1}{k_1}$ is also invariant under $\Gmax_{P_1}$. \label{item:also-inv} 
\item The correlator  \vspace*{-1ex}
\begin{equation}
 \left< \kn{m_1}{g_1}, \kn{n_1}{h_1}, \kn{p_1}{k_1} \right>^{P_1, \Gmax_{P_1}} \label{eq:break-off}
\end{equation}
\end{enumerate}
is non-vanishing. \label{item:break-off}
\end{proposition}
\begin{proof}
Condition \ref{item:also-inv} follows from $\Gmax$-invariance (\autoref{gmax-inv}).
If this condition is satisfied, then we can pick groups $H_1 = \Gmax_{P_1}$ and $H_2 = \pi_2(H)$, where $\pi_2$ is projection onto the second factor of $\Gmax_{P} \cong \Gmax_{P_1} \times  \Gmax_{P_2}$.  We can then check that \autoref{lem:change-group} applies, and we can compute \eqref{eq:breakable-corr} in the ring $\HH_{W_1 + W_2, H_1 \times  H_2}$.  \autoref{ax:sums} now shows \eqref{eq:breakable-corr} is the product of \eqref{eq:break-off} and 
$  \left< \kn{m_2}{g_2}, \kn{n_2}{h_2}, \kn{p_2}{k_2} \right>^{P_2, H_2}$.\end{proof}

\section{Algebra homomorphism}

Although we have a well-defined mirror map, we must still prove that it respects multiplication.  We do this first for fundamental factors and then in the general case.  This requires several technical lemmas about loops and chains, which are provided in an appendix to this paper.

\subsection{Fundamental factors}

\begin{proposition}\label{fundhomom}
If $(W,G)$ is a fundamental pair, then the mirror map 
$$\mmap:\BB_{W\tr,G\tr} \rightarrow  \HH_{W,G}$$
defined in \autoref{fundamentalmmap} is a ring homomorphism
\end{proposition}
\begin{proof}
Assume that $\kn{m}{g}$ and $\kn{m'}{g'}$ are elements of $\BB_{W\tr,G\tr}$.  In order to check that $\mmap(\kn{m}{g}\starb \kn{m'}{g'}) = \mmap(\kn{m}{g})\stara \mmap(\kn{m'}{g'})$, we must consider three distinct cases: 
\begin{enumerate}
\item  $g=g'=0$ 
\item  $g=0$ but $g'\neq0$ and 
\item $g\neq 0$ and $g'\neq0$.
\end{enumerate}
In case (1) restrict the map $\mmap$  to the untwisted sector $\BB_{0}$, where it is a rescaling of the ring homomorphism defined in \cite{krawitz}.

Case (2) breaks into two parts.  Firstly, in the case of $\kn{1}{0}\starb \kn{1}{g'}$, the element $\kn{1}{0}$ is the identity element of $\BB_{W\tr,G\tr}$ and it maps to the identity element $\ringid = \kn{1}{J}$ of $\HH_{W,G}$, so the product is preserved in this case.  

Secondly, in the case of $\kn{m}{0}\starb \kn{1}{g'}$, where $m$ is not constant, \autoref{lem:vanish-as-well} (proved in the next section) shows that both products vanish.

In case (3) write $W$ as a decoupled sum of atomic types.  By \autoref{lem:write-sym-groups} we may write $g$ and $g'$ uniquely as  $g = (\M\tr)^{-1}(\brr + \one)$ and $g' = (\M\tr)^{-1}(\bs + \one)$, respectively, with the components of $\brr$ and $\bs$ satisfying the constraints of that lemma, as determined by the corresponding atomic types.

If $[\M^{-1}(\brr + \bs+ \mathbf 2)] \neq 0$ then the B-side product vanishes by definition, and the A-side product vanishes by \autoref{cor:equiv}.

%
%

 If $[\M^{-1}(\brr + \bs+ \mathbf 2)] = 0$, then 
$\brr + \bs = \expon - \one$.  For loops this follows from \autoref{rem:1of3}, 
since if either \eqref{enum:odd} or \eqref{enum:even} were true the B-side product would not be a product of non-identity sectors. 
Similarly for Fermats, $r_i + s_i = a_i -1$.  For chains, it is easy to check that if $\brr+\bs \neq \expon-\one$ then we get a contradiction to \autoref{property}.
By definition, the B-side product is 
\begin{equation}
\kn{1}{(\M\tr)^{-1}(\brr + \one)} \starb \kn{1}{(\M\tr)^{-1}(\bs + \one)}
=	\frac1\mu_{W\tr}\kn{\hess(W\tr)}{0}. \label{eq:b-side-product-is-by-def}
\end{equation}
To compute the image $\mmap(\knb{\hess(W\tr)}{0})$ we write\footnote{Recall that $\mmap$ agrees with Krawitz' algebra homomorphism for these correlators.} 
$$\mmap\left(\frac{1}{\mu_{W\tr}}\knb{\hess(W\tr)}{0}\right) = \rcb^{-\bwb\cdot(\bbeta+\bgamma)} \mmap\left(\knb{\Y^{\bbeta}}{0}\right) \starb \left(\knb{\Y^{\bgamma}}{0}\right)$$
with $\bbeta+\bgamma = \expontr$.  
We may apply \cite[Lem 3.2 and 3.5]{krawitz}, which states that 
\begin{equation}
	\kn{1}{[\M^{-1}(\bbeta + \one)]} \stara \kn{1}{[\M^{-1}(\bgamma + \one)]} = \kn{1}{[\M^{-1}(\bbeta + \bgamma + \one)]} \label{eq:narrow-mul}
\end{equation}
as long as $\bbeta + \bgamma \preceq \ba- \one$ (componentwise) 
and $[\M^{-1}(\bbeta + \bgamma + \one)] \neq 0$. 
Combining with \autoref{lem:-J} gives 
\begin{equation}
\mmap\left( \frac{1}{\mu_{W\tr}}\knb{\hess(W\tr)}{0}\right) = 
	\kn{1}{{[(\M)^{-1}\expontr]} }  = \kna{1}{-J}. \label{eq:product-maps-to}
\end{equation}
On the other hand, the A-side product is 
\begin{equation*}
\rca^{-\bwa\cdot(\brr+\bs)} \pair{ \kn{ \X^{\brr}}{\mathbf 0} }{ \kn{\X^{\bs}}{\mathbf 0}}
\kna{1}{-J}  = \kna{1}{-J}.  \label{eq:A-side-product2}
\end{equation*}
\end{proof}

\subsection{Mixed products}
The following lemma finishes the proof of \autoref{fundhomom}.
We must prove this lemma in more generality than just for fundamental pairs, since we need to apply it to pairs of elements which cannot be split nicely by a common partition. 

We will continue to use the notation of \autoref{BasicSetup}
and we will write the vector of variables of $W_j\tr$ as $\Y_j$ and similarly the vector of variables of $W_j$ as $\X_j$.
\begin{lem}  \label{lem:vanish-as-well}
Suppose we have a pair $\knb{\prod m_i}{\dse{g_i}}$ and $\knb{\prod n_i}{\dse{h_i}}$ and for some $j$, we have $m_j = 1$, $g_j = [(\M_j\tr)^{-1}(\bs + 1)] \neq 0$, and $n_j = \Y_j^{\bbeta} \neq 1$, $h_j = 0$.      Then the following products both vanish.
\begin{equation*}
	\kn{\prod m_i}{\dse{g_i}} \starb \kn{\prod n_i}{\dse{h_i}} 
\dsand
	\mmap(\kn{\prod m_i}{\dse{g_i}}) \stara \mmap(\kn{\prod n_i}{\dse{h_i}}) \end{equation*}
\end{lem}


\begin{proof}[Proof of \autoref{lem:vanish-as-well}]
Since $g_j + h_j \neq 0$ but $\Y_j^{\bbeta} \neq 1$, the product $\knb{\prod m_i}{\dse{g_i}} \starb \kn{\prod n_i}{\dse{h_i}} $ will vanish. It remains to see that $	\mmap(\kn{\prod m_i}{\dse{g_i}}) \stara \mmap(\kn{\prod n_i}{\dse{h_i}})$  vanishes as well.  The image of $\kn{\prod  m_i}{ \dse{g_i}}$ will be a linear combination of things of the form
$
\kn{\prod \hat m_i}{\dse{\hat g_i}},
$
 where $\hat m_j = \X_j^{\bs}$, $\hat g_j = 0$. 

 We can see that the image of $\kn{\prod n_i}{\dse{h_i}}$ will be a linear combination of things of the form $\kn{\prod \hat n_i}{\hat h_i}$
\autoref{lem:write-group} shows exactly when $\hat h_j = 0$. If $W_j$ is an even variable loop, $\hat h_j = 0$ only when
$\beta_k = \delta_{\text{odd}}^k (a_k -1)$ or $\beta_k = \delta_{\text{even}}^k (a_k -1)$, which implies that
 $n_j$ is $\prod_{k \text{ even/odd}} X_{j,k}^{a_{j,k}-1}$.  
 For the potentially non-vanishing correlators, we can apply \autoref{lem:break-off} and examine pieces of the form
\begin{equation}
	\left<\kn{\X_j^{\bs}}{0}, \kn{\prod_{k \text{ odd/even}} X_{j,k}^{a_{j,k}-1}}{0}, \kn{1}{J_{W_j}} \right> \label{eq:vaw1},
\end{equation}
where we filled in the third spot using \autoref{Jg-grading}.  We see that this is non-vanishing only if $\X_j^{\bs}$ pairs with $\prod_{k \text{ odd/even}} X_{j,k}^{a_{j,k}-1}$, which will only happen if $\bs$ is of the form $s_k = \delta_{\text{odd/even}}^k (a_k -1)$.  But in that case we have $[(\M_j\tr)^{-1} (\bs + \one)] = 0$, contradicting our assumption.
For the case that $W_j$ is a chain, 
a similar argument produces a contradiction to the  Chain \autoref{chainproperty}; thus the product vanishes.

Suppose now that $\bbeta$ is not of these special forms.  Then $\hat n_j = 1$, $\hat h_j = [(\M\tr_j)^{-1}(\bbeta + \one)] \neq 0$.  Again using \autoref{lem:break-off}  and \autoref{Jg-grading}, we can examine pieces of the form 
\begin{equation}
 \left<\kn{\X_j^{\bs}}{0}, \kn{1}{[(\M_j)^{-1}(\bbeta + \one)]}, \kn{1}{[-{\M_j}^{-1}\bbeta]} \right> \label{eq:vaw2}.
\end{equation}
It follows from \autoref{lem:write-sym-groups} \eqref{item:write-loop-group}
 that $(\M_j)^{-1}\bbeta \neq 0$. 
Applying \autoref{rem:gmax-inv} to $\bar{\rho}_{j,k}$  this correlator vanishes unless $\bar{\rho}_{i,k}\tr(\bs + \one) \in \ZZ$. But this would imply that $[(\M_j\tr)^{-1}(\bs + \one)]=0$, again contradicting the hypothesis.
\end{proof}

\subsection{The general case}

We can now prove the homomorphism property in the general case.
\begin{theorem}\label{genhomom}
If $(W,G)$ satisfies \autoref{property} then 
then the mirror map 
$\mmap:\BB_{W\tr,G\tr} \rightarrow  \HH_{W,G}
$
is a ring homomorphism
\end{theorem}
Given a general pair of elements $\kn{\prod m_i}{\dse{g_i}}$ and  $\kn{\prod n_i}{\dse{h_i}}$ in $\BB_{W\tr,G\tr}$, if they are not of the form described in \autoref{lem:vanish-as-well}, then we will show that we can pick a partition $\II$ and groups $\{G_I\}$ such that both elements split nicely with respect to $\II$ and $\{G_I\}$.  Since the algebra $\BB_{\sum_{I\in \II}{W_I}, \bigoplus_{I\in \II} G\tr_I}$ is a tensor product of fundamental factors, the desired result will follow.

\begin{defn} \label{def:I-g}\label{def:G-I} 
 Excluding the cases described in \autoref{lem:vanish-as-well}, consider a pair of  elements  $\kn{\prod m_i}{\dse{g_i}}$ and $ \kn{\prod n_i}{\dse{h_i}}$ in $\BB_{W\tr,G\tr}$.
Define the partition $
    \II = \{I_g, I_h, I_{g,h}\} \cup \{\{i\}\}_{i \in I_0}.
 $ as follows.
\begin{itemize}
	\item Let $I_h$ be the set of indices such that $g_i = 0$ and $h_i \neq 0$ (then by assumption $m_i = 1$).
	\item Let $I_g$ be the set of indices such that $h_i= 0$ and $g_i \neq 0$ (then $n_i = 1$).
	\item Let $I_{g,h}$ be the set of indices where $g_i, h_i \neq 0$.
	\item Let $I_0$ be the set of indices such that $g_i = h_i = 0$.
\end{itemize}

%
%

Also define groups as follows:
\begin{itemize}
	\item Let $G_{I_h}\tr$ be the group of symmetries of $W_{I_h}$ generated by $h_{I_h}$.
	\item Let $G_{I_g}\tr$ be the group of symmetries of $W_{I_g}$ generated by $g_{I_g}$.
	\item Let $G_{I_{g,h}}\tr$ be the group of symmetries of $W_{I_{g,h}}$ generated by both $h_{I_{g,h}}$ and $g_{I_{g,h}}$.
	\item Let $G_i\tr$ be the trivial group of symmetries of $W_i$ for $i \in I_0$.
\end{itemize}
\end{defn}

\begin{lem} \label{lem:newgroups-in-SL}
The groups $\{ G_I\tr \}$ 
are each contained in $\SL$.  Both $\kn{\prod m_i}{\dse{g_i}}$ and $\kn{\prod n_i}{\dse{h_i}}$ split nicely (see \autoref{def:nicely}) with respect to $\II$ and $\{ G_I\tr \}$. 
\end{lem}
This follows from the from the definitions, using the the fact that $\dse{g_i}$ preserves $\kn{\prod n_i}{\dse{h_i}}$ and vice versa.
\section{The Pairing} 

Having established the algebra isomorphism, 
to complete the proof of \autoref{mainthm}
we need only check that the pairing is preserved by $\mmap$.  
\begin{proposition}\label{thm:pairing}
If $(W,G)$ satisfy \autoref{property}, then the mirror map $\mmap:\BB_{W\tr,G\tr} \rightarrow  \HH_{W,G}$ preserves the pairing and therefore is an isomorphism of Frobenius algebras.
\end{proposition}

\begin{proof}
Assume that $\kn{m}{g}$ and $\kn{m'}{g'}$ are elements of $\BB_{W\tr,G\tr}$.  We must check that 
\begin{equation}\label{pairingcheck}
\left<\mmap(\kn{m}{g}, \kn{m'}{g'})\right>_B = \left<\mmap(\kn{m}{g}), \mmap(\kn{m'}{g'})\right>_A.
\end{equation}
\begin{lem} \label{lem:mixedpairing}For any mixed pair, as described in the hypothesis of \autoref{lem:vanish-as-well},  both pairings in Equation~\eqref{pairingcheck} vanish.
\end{lem}

\begin{proof}
On the B-side this is immediate, since $g+g' \neq [0]$.  The form of their images on the A-side was computed explicitly in the proof of \autoref{lem:vanish-as-well}, and we saw in \eqref{eq:vaw1} that if the pairing were non-trivial, it would violate the hypothesis.  In \eqref{eq:vaw2}, we can see that the images do not come from inverse sectors.
\end{proof}
Now, by using the same partitions we used for the verification of the homomorphism property, we can reduce to the case of a fundamental pair $(W,G)$.  
Furthermore, it suffices to check for each factor in the tensor product.  We again have three cases to check: (1) $g=g'=0$,  (2)  $g=0$ but $g'\neq0$ and (3) $g\neq 0$ and $g'\neq0$.

Case (2) follows immediately from \autoref{lem:mixedpairing}.  Cases (1) and (3) are straightforward computations, very similar to the computations done in the proof of the third case of \autoref{fundhomom}.
\end{proof}

\section{Appendix: Some loop and chain lemmas}

In the following  proofs, unless otherwise indicated, we assume that $W$ is a atomic polynomial with variables ordered as in \autoref{sec:classification} and with exponent matrix $\M$. We take $\expon$ as defined in \autoref{lem:milnor}.  When discussing loop polynomials, we take the indices modulo $N$.

%
%
The first lemma follows immediately from the definitions.
\begin{lem} \label{lem:-J}
Let $J$ be the corresponding exponential grading element corresponding to $W$ (see \autoref{def:rho-and-J}).  
Then
\[
-J = \left[\M^{-1} \expontr\right] \in (\QQ/\ZZ)^N.
\]
\end{lem}

\begin{lem} \label{lem:hess}
If $\X^\bt \in \QQQ_{W}$ is a scalar multiple of the Hessian, then $[(\M\tr)^{-1}(\bt + \mathbf 2)] = 0$.
\end{lem}
\begin{proof}
If $\bt = \expon -\one$, then the result follows from \autoref{lem:-J}.  
The more general case is a straightforward computation once one translates the Jacobian relations into an operation on exponent vectors.
\end{proof}
%
\begin{lem} \label{lem:srAn}
Suppose that
$	\bv = \M\tr    \bn, $
where $\bn$ has integer entries and the entries of $\bv$ satisfy
$2 \le v_i \le 2a_i.  $
\begin{enumerate}
	\item\label{loop-delta} If $W$ is a loop type polynomial, then one of the following is true:
		\begin{enumerate}
			\item $\bv = \expon + \one.$
			\item $v_i = \delta_\text{odd}^i (2a_i-2)+2.$ \label{item:odd}
			\item $v_i = \delta_\text{even}^i (2a_i -2)+2.$ \label{item:even}
		\end{enumerate}
		The latter two cases can occur only if $N$ is even.
	\item If $W$ is a chain type polynomial, then
	there is some $s \in \{1, \ldots, N \}$ where $n_s$ is the 
	first coordinate of $\bn$ 
	equal to $1$, and 
\[
v_i = \left\{\begin{array}{l r} \delta_\text{odd}^i (2a_i -2) + 2	   & i<s \\
		a_s + 2\delta_\text{even}^s & i=s\\
	 	1+a_i& i>s\end{array}\right.
\]
\end{enumerate}
\end{lem}
\begin{proof}
Consider first $W$ a loop. We have the set of inequalities
\begin{equation}
 2 \le a_i n_i + n_{i-1} \le 2a_i \label{ineq1}.
\end{equation}
If $n_j \le 0$ for some $j$, then we see that $n_{j-1} \ge 2$.
Then $a_{j-1}2 + n_{j-2} \le 2a_{j-1}$
which implies that $n_{j-2} \le 0$. Thus we see that if any $n_i$ is not 1, then the entries of the vector $\bn$ must alternate being less than or equal to zero and being greater than or equal to 2.  This is impossible if $N$ is odd.  

Suppose now that for some $j$, $n_j \ge m$, where $m \ge 3$ is an integer.  
Using the inequality  $a_jm + n_{j-1} \le 2a_i$ repeatedly, one shows that 
we can find an entry of $\bn$ larger than any natural number, a contradiction.  
Thus $n_i \le 2$ for all $i$, and it then follows that $n_i \ge 0$.  Thus,  either $\bn = \one$, $\bn = [2,0,2,0, \dots ,2,0]\tr$, or $\bn = [0,2,0,2, \dots, 0, 2]\tr$, and the latter two cases can occur only when $n$ is even.  Then the result follows from computing $\M\tr\bn$.

To prove the statement about chain polynomials, start by considering the possibilities for $n_1$, and it is straightforward to check that $\bn$ is of the form $[2,0, \dots , 2, 0, 1, \dots 1]$  or $ [2, 0, \dots, 0,2,1,\dots, 1]$, from which the result follows. \end{proof}


\begin{remark} \label{rem:1of3} 
\mbox{}
\begin{enumerate}
    \item\label{rem:loop1of3} Suppose that for a loop polynomial $0 \le r_i, s_i \le a_i-1$ (as in  \autoref{lem:write-sym-groups} \eqref{item:write-loop-group}) and $[(\M\tr)^{-1}(\brr + \bs + \mathbf 2)] = 0$. The vector $\bv = \brr + \bs + \mathbf 2$ satisfies the hypothesis of \autoref{lem:srAn} and one of the following is true:
        \begin{enumerate}
            \item $\brr + \bs = \expon - \one$
            \item $r_i = s_i = \delta_\text{odd}^i (a_i-1)$ \label{enum:odd}
            \item $r_i = s_i  = \delta_\text{even}^i (a_i -1)$ \label{enum:even}
        \end{enumerate}
    \item\label{rem:chain1of3} Suppose that for a chain polynomial $\brr$ and $\bs$ satisfy the Chain \autoref{chainproperty}
        and $[(\M\tr)^{-1}(\brr + \bs + \mathbf 2)] = 0$.  By Lemma \ref{lem:srAn} either $\brr+\bs = \expon -1$ or
        \[
            (\brr+\bs)_i = \left\{\begin{array}{l r}    \delta_\text{odd}^i (2a_i -2)& i<s \\
                    a_s - 2\delta_\text{odd}^s & i=s\\
                     a_i-1 & i>s 	\end{array}\right.
        \]
\end{enumerate}
\end{remark}


\begin{lem} \label{lem:2a-2}
If $W$ is a loop and has an even number of variables, then in $\QQQ_{W}$ we have
\[
	\prod_{i\text{ odd}} X_i^{2a_i - 2} = \prod_{i\text{ even}} (-a_{i}) \prod_{i=1}^N X_i^{a_i-1}
\dsand	\prod_{i\text{ even}} X_i^{2a_i - 2} = \prod_{i\text{ odd}} (-a_{i}) \prod_{i=1}^N X_i^{a_i-1}.
\]
(Thus these are also multiples of $\hess(W)$.)
\end{lem}
\begin{proof}
The Jacobian relations for a loop polynomial are $X_i^{a_i} = - a_{i+1}X_{i+1}^{a_{i+1}-1}X_{i+2}$.  Repeated application of this relation gives the result.
\end{proof}

\begin{cor} \label{cor:equiv} 
Suppose that $W = \sum W_i$ is a sum of Fermat, loop, and chain type polynomials with exponent matrix $\M$, and that $a_{i,j}$ are the defining exponents of $W_i$.
Suppose $\X^{\brr}$ and $\X^{\bs}$ are representatives of elements of $\QQQ_{W}$ with the exponent vectors $\brr$, $\bs$ satisfying the conditions of  \autoref{lem:write-sym-groups}).
Then the following are equivalent: 
\begin{enumerate}
	\item $[(\M\tr)^{-1}(\brr + \bs+ \mathbf 2)] = 0$ \label{first}
	\item $\X^{\brr + \bs}$ is a scalar multiple of the Hessian. \label{second}
\end{enumerate}
\end{cor}
\begin{proof}
It suffices to prove this for atomic types.  For the Fermat type, it is obvious.
For the loop and chain types, \eqref{second} implies \eqref{first} is the content of \autoref{lem:hess}.
To see that \eqref{first} implies \eqref{second}, first use  \autoref{rem:1of3}.  Then, for loops, apply \autoref{lem:2a-2}, if necessary, to show that $\X^{\brr + \bs}$ is a scalar multiple of $\X^{\expon-\one}$.
For chains, if $\brr + \bs \neq \expon -1$, and $s$ is odd, then apply the Jacobian relation 
$X_i^{a_i} = -a_{i+1}X_{i+1}^{a_{i+1}-1}X_{i+2}$ for $i = s-2, s-4, \dots, 1$ to show that 
$\X^{\brr + \bs} = \prod_{\substack{i\text{ even} \\ i < s}}(-a_i) \X^{\expon-\one}.$
If $s$ is even, apply the Jacobian relation $X_i^{a_i-1}X_{i+1} = -\frac{1}{a_{i}}X_{i-1}^{a_{i-1}}$ for $i = s-1, s-3, \dots, 3$ to show that 
$\X^{\brr + \bs} = \prod_{\substack{ i\text{ odd} \\ i < s}}(-\frac{1}{a_i}) \X^{\expon-\one}\cdot X_1^{a_1}X_2^{a_2} = 0.$
Either way, at least one of $\brr$ or $\bs$ violates the Chain \autoref{chainproperty}.
\end{proof}


\bibliographystyle{amsalpha}

\bibliography{references2}

\end{document}